\documentclass[12pt, reqno]{amsart}
\usepackage{ amsmath,amsthm, amscd, amsfonts, amssymb, graphicx, color}
\usepackage[bookmarksnumbered, colorlinks, plainpages]{hyperref}
\newtheorem{thm}{Theorem}[section]
\newtheorem{cor}[thm]{Corollary}
\newtheorem{lem}[thm]{Lemma}

\newtheorem{exam}[thm]{Example}
\numberwithin{equation}{section}

\begin{document}

\title[Matrices Over Zhou nil-clean Rings]{Matrices Over Zhou nil-clean Rings}

\author{Marjan Sheibani}
\author{Huanyin Chen}
\address{
Faculty of Mathematics\\Statistics and Computer Science\\  Semnan University\\ Semnan, Iran}
\email{<m.sheibani1@gmail.com>}
\address{
Department of Mathematics\\ Hangzhou Normal University\\ Hang -zhou, China}
\email{<huanyinchen@aliyun.com>}


\begin{abstract}
A ring $R$ is Zhou nil-clean if every element in $R$ is the sum of two tripotents and a nilpotent that commute. Let $R$ be a Zhou nil-clean ring. If $R$ is 2-primal (of bounded index), we prove that every square matrix over $R$ is the sum of two tripotents and a nilpotent. These provides a large class of rings over which every square matrix has such decompositions by tripotent and nilpotent matrices.\end{abstract}

\maketitle

\section{Introduction}

Throughout, all rings are associative with an identity. A ring $R$ is strongly nil-clean provided that every element in $R$ is the sum of an idempotent and a nilpotent that commutate (see~\cite{D}). A ring $R$ is strongly weakly nil-clean provided that every element in $R$ is the sum or difference of a nilpotent and an idempotent that commutate (see~\cite{CS2}). An element $a$ in a ring $R$ is tripotent if $a=a^3$.
A ring $R$ is strongly 2-nil-clean if every element in $R$ is the sum of a tripotent and a nilpotent that commute (see~\cite{CS}). The subjects of strongly
and strongly weakly nil-clean rings are interested for so many mathematicians, e.g., ~\cite{B,DW,K,KWZ} and ~\cite{KZ}.

Very recently, Zhou investigated a class of rings in which elements are the sum of two tripotents and a nilpotent that commute (see~\cite{Z}). For the conventience, we call such ring a Zhou nil-clean ring. The purpose of this paper is to explore matrices over Zhou nil-clean rings.
A ring $R$ is 2-primal if its Baer-McCoy radical (i.e., it is equal to the intersection of all prime ideals in $R$)
coincides with the set of nilpotents in $R$. For instances, every commutative (reduced) ring is 2-primal. A ring $R$ is of bounded index if there exists $m\in {\Bbb N}$ such that $x^m=0$ for all nilpotent $x\in R$.
Let $R$ be Zhou nil-clean. If $R$ is 2-primal (of bounded index), we prove that every square matrix over $R$ is the sum of two tripotents and a nilpotent.
These provides a large class of rings over which every square matrix has such decompositions by tripotent and nilpotent matrices.

We use $N(R)$ to denote the set of all nilpotent elements in $R$ and $J(R)$ the Jacobson radical of $R$.
${\Bbb N}$ stands for the set of all natural numbers.

\section{Zhou Rings}

A ring $R$ is a Zhou ring if every element in $R$ is the sum of two tripotents that commute. The structure of Zhou rings was studied in~\cite{KW}.
This section is devoted to preliminary observations concerning on matrices over Zhou rings which will be used in the sequel. We begin with

\begin{lem} Every square matrix over ${\Bbb Z}_3$ is the sum of two idempotents and a nilpotent.\end{lem}
\begin{proof}  As every
matrix over a field has a Frobenius normal form, we may assume
that $A=\left(
\begin{array}{cccccc}
0&&&&&c_0\\
1&0&&&&c_1\\
&1&0&&&c_2\\
&&&\ddots&&\vdots\\
&&&\ddots&0&c_{n-2}\\
&&&&1&c_{n-1}
\end{array}
\right).$

Case I. $c_{n-1}=0$. Choose $$W=\left(
\begin{array}{cccccc}
0&&&&&0\\
1&0&&&&0\\
&1&0&&&0\\
&&&\ddots&&\vdots\\
&&&\ddots&0&0\\
&&&&1&0
\end{array}
\right), E_1=\left(
\begin{array}{cccccc}
0&&&&&c_0\\
0&0&&&&c_1\\
&0&0&&&c_2\\
&&&\ddots&&\vdots\\
&&&\ddots&0&c_{n-2}\\
&&&&0&1
\end{array}
\right),$$ $$E_2=\left(
\begin{array}{cccccc}
0&&&&&0\\
0&0&&&&0\\
&0&0&&&0\\
&&&\ddots&&\vdots\\
&&&\ddots&0&0\\
&&&&0&-1
\end{array}
\right).$$ Then $E_1^3=E_1$ and $E_2^3=E_2$, and so $A=E_1+E_2+W$.

Case II. $c_{n-1}=1$. Choose $$W=\left(
\begin{array}{cccccc}
0&&&&&0\\
1&0&&&&0\\
&1&0&&&0\\
&&&\ddots&&\vdots\\
&&&\ddots&0&0\\
&&&&1&0
\end{array}
\right), E=\left(
\begin{array}{cccccc}
0&&&&&c_0\\
0&0&&&&c_1\\
&0&0&&&c_2\\
&&&\ddots&&\vdots\\
&&&\ddots&0&c_{n-2}\\
&&&&0&1
\end{array}
\right).$$ Then $E^3=E$, and so $A=E+0+W$.

Case III. $c_{n-1}=2$. Choose $$\begin{array}{c}
W=\left(
\begin{array}{cccccc}
0&&&&&0\\
1&0&&&&0\\
&1&0&&&0\\
&&&\ddots&&\vdots\\
&&&\ddots&0&0\\
&&&&1&0
\end{array}
\right), E_1=\left(
\begin{array}{cccccc}
0&&&&&c_0\\
0&0&&&&c_1\\
&0&0&&&c_2\\
&&&\ddots&&\vdots\\
&&&\ddots&0&c_{n-2}\\
&&&&0&1
\end{array}
\right),\\
E_2=\left(
\begin{array}{cccccc}
0&&&&&0\\
0&0&&&&0\\
&0&0&&&0\\
&&&\ddots&&\vdots\\
&&&\ddots&0&0\\
&&&&0&1
\end{array}
\right).
\end{array}$$ Then $E_1^3=E_1$ and $E_2^3=E_2$, and so $A=E_1+E_2+W$.
 \end{proof}

\begin{lem} Every square matrix over ${\Bbb Z}_5$ is the sum of two tripotents and a nilpotent.\end{lem}
\begin{proof} As every
matrix over a field has a Frobenius normal form, we may assume
that $A=\left(
\begin{array}{cccccc}
0&&&&&c_0\\
1&0&&&&c_1\\
&1&0&&&c_2\\
&&&\ddots&&\vdots\\
&&&\ddots&0&c_{n-2}\\
&&&&1&c_{n-1}
\end{array}
\right).$

Case I. $c_{n-1}=0$. Choose $$\begin{array}{c}
W=\left(
\begin{array}{cccccc}
0&&&&&0\\
1&0&&&&0\\
&1&0&&&0\\
&&&\ddots&&\vdots\\
&&&\ddots&0&0\\
&&&&1&0
\end{array}
\right), E_1=\left(
\begin{array}{cccccc}
0&&&&&c_0\\
0&0&&&&c_1\\
&0&0&&&c_2\\
&&&\ddots&&\vdots\\
&&&\ddots&0&c_{n-2}\\
&&&&0&1
\end{array}
\right),\\
E_2=\left(
\begin{array}{cccccc}
0&&&&&0\\
0&0&&&&0\\
&0&0&&&0\\
&&&\ddots&&\vdots\\
&&&\ddots&0&0\\
&&&&0&-1
\end{array}
\right).
\end{array}$$ Then $E_1^3=E_1$ and $E_2^3=E_2$, and so $A=E_1+E_2+W$.

Case II. $c_{n-1}=1$. Choose $$W=\left(
\begin{array}{cccccc}
0&&&&&0\\
1&0&&&&0\\
&1&0&&&0\\
&&&\ddots&&\vdots\\
&&&\ddots&0&0\\
&&&&1&0
\end{array}
\right), E=\left(
\begin{array}{cccccc}
0&&&&&c_0\\
0&0&&&&c_1\\
&0&0&&&c_2\\
&&&\ddots&&\vdots\\
&&&\ddots&0&c_{n-2}\\
&&&&0&1
\end{array}
\right).$$ Then $E^3=E$, and so $A=E+0+W$.

Case III. $c_{n-1}=-1$. Choose $$W=\left(
\begin{array}{cccccc}
0&&&&&0\\
1&0&&&&0\\
&1&0&&&0\\
&&&\ddots&&\vdots\\
&&&\ddots&0&0\\
&&&&1&0
\end{array}
\right), E=\left(
\begin{array}{cccccc}
0&&&&&c_0\\
0&0&&&&c_1\\
&0&0&&&c_2\\
&&&\ddots&&\vdots\\
&&&\ddots&0&c_{n-2}\\
&&&&0&-1
\end{array}
\right).$$ Then $E^2=-E$, and so $E^3=E$ and $A=E+0+W$.

Case IV. $c_{n-1}=2$. Choose $$\begin{array}{c}
W=\left(
\begin{array}{cccccc}
0&&&&&0\\
1&0&&&&0\\
&1&0&&&0\\
&&&\ddots&&\vdots\\
&&&\ddots&0&0\\
&&&&1&0
\end{array}
\right), E_1=\left(
\begin{array}{cccccc}
0&&&&&c_0\\
0&0&&&&c_1\\
&0&0&&&c_2\\
&&&\ddots&&\vdots\\
&&&\ddots&0&c_{n-2}\\
&&&&0&1
\end{array}
\right),\\
E_2=\left(
\begin{array}{cccccc}
0&&&&&0\\
0&0&&&&0\\
&0&0&&&0\\
&&&\ddots&&\vdots\\
&&&\ddots&0&0\\
&&&&0&1
\end{array}
\right).
\end{array}$$ Then $E_1^3=E_1$ and $E_2^3=E_2$, and so $A=E_1+E_2+W$.

Case IV. $c_{n-1}=-2$. Choose $$\begin{array}{c}
W=\left(
\begin{array}{cccccc}
0&&&&&0\\
1&0&&&&0\\
&1&0&&&0\\
&&&\ddots&&\vdots\\
&&&\ddots&0&0\\
&&&&1&0
\end{array}
\right), E_1=\left(
\begin{array}{cccccc}
0&&&&&c_0\\
0&0&&&&c_1\\
&0&0&&&c_2\\
&&&\ddots&&\vdots\\
&&&\ddots&0&c_{n-2}\\
&&&&0&-1
\end{array}
\right),\\
E_2=\left(
\begin{array}{cccccc}
0&&&&&0\\
0&0&&&&0\\
&0&0&&&0\\
&&&\ddots&&\vdots\\
&&&\ddots&0&0\\
&&&&0&-1
\end{array}
\right).
\end{array}$$ Then $E_1^3=E_1$ and $E_2^3=E_2$, and so $A=E_1+E_2+W$.

Therefore we complete the proof.\end{proof}

Recall that a ring $R$ is a Yaqub ring if it is the subdirect product of ${\Bbb Z}_3$'s. A ring $R$ is a Bell ring if it is the subdirect product of ${\Bbb Z}_5$'s. We have

\begin{lem} Every Zhou ring is isomorphic to a strongly nil-clean ring of bounded index, a Yaqub ring, a Bell ring or products of such rings.
\end{lem}
\begin{proof} Let $R$ be a Zhou ring. In view of~\cite[Theorem 5.2]{KW}, $R$ is a ring $R_1$, a Yaqub ring $R_2$, a Bell ring $R_3$ or products of such rings. Here, $R_1/J(R_1)$ is Boolean and $U(R_1)$ is a group of exponent $2$. Let $x\in J(R_1)$. By the proof of ~\cite[Theorem 5.2]{KW}, $x^6=x^4$, and so
$x^4=0$. Thus, $J(R)^4=0$. It follows by ~\cite[Theorem 2.7]{KWZ}, $R_1$ is a strongly nil-clean ring of bounded index$4$.\end{proof}

\begin{thm} Let $R$ be a Zhou ring. Then every square matrix over $R$ is the sum of two tripotents and a nilpotent.
\end{thm}
\begin{proof} In view of Lemma 2.3, $R$ is isomorphic to $R_1, R_2, R_3$ or the products of these rings, where $R_1$ is a strongly nil-clean ring of bounded index, $R_2$ is
a Yaqub ring and $R_3$ is a Bell ring.

Step 1. Let $A\in M_n(R_1)$. In view of~\cite[Corollary 6.8]{KWZ}, there exist an idempotent $E\in M_n(R_1)$ and $W\in N(M_n(R_1))$ such that $A=E+W$.

Step 2. Let $A\in M_n(R_2)$, and
let $S$ be the subring of $R_2$ generated by the entries of $A$.
That is, $S$ is formed by finite sums of monomials of the form:
$a_1a_2\cdots a_m$, where $a_1,\cdots,a_m$ are entries of $A$. Since $R_2$ is a
commutative ring in which $3=0$, $S$
is a finite ring in which $x=x^3$ for all $x\in S$. Thus, $S$ is isomorphic to finite direct
product of ${\Bbb Z}_3$. As $A\in M_n(S)$, it follows by Lemma 2.1 that $A$ is the sum of two tripotents and a nilpotent matrix over $S$.

Step 3. Let $A\in M_n(R_3)$, and
let $S$ be the subring of $R_3$ generated by the entries of $A$.
Analogously, $S$ is isomorphic to finite direct
product of ${\Bbb Z}_5$. As $A\in M_n(S)$, it follows by Lemma 2.2 that $A$ is the sum of two tripotents and a nilpotent matrix over $S$.

Let $A\in M_n(R)$. We may write $A=(A_1,A_2,A_3)$ in $M_n(R_1)\times M_n(R_2)\times M_n(R_3)$, where $A_1\in M_n(R_1), A_2\in M_n(R_2),A_3\in M_n(R_3)$. According to the preceding discussion,
we easily complete the proof.\end{proof}

\begin{cor} Let $R$ in which for any $x\in R$, $x^5=x$. Then every square matrix over $R$ is the sum of two tripotents and a nilpotent.\end{cor}
\begin{proof} In view of ~\cite[Theorem 2.11]{Z}, every element in $R$ is the sum of two tripotents and a nilpotent.
But $N(R)=0$, and so $R$ is a Zhou ring. This completes the proof, by Theorem 2.4.\end{proof}

\section{Zhou nil-clean Rings}

The goal of this section is to explore matrix decompositions for Zhou nil-clean rings. The following lemma is crucial.

\begin{lem} Let $R$ be a ring. Then the following are equivalent:
\end{lem}
\begin{enumerate}
\item [(1)] {\it $R$ is Zhou nil-clean.}
\vspace{-.5mm}
\item [(2)] {\it $R/J(R)$ is a Zhou ring and $J(R)$ is nil.}
\end{enumerate}
\begin{proof} $(1)\Rightarrow (2)$ In view of~\cite[Theorem 2.11]{Z}, $J(R)$ is nil and $R/J(R)$ has the identity $x=x^5$.
By using~\cite[Theorem 2.11]{Z} again, every element in $R/J(R)$ is the sum of two tripotents and a nilpotent $\overline{w}$. Clearly,
$R/J(R)$ is reduced, and so $\overline{w}=\overline{0}$. Thus, $R/J(R)$ is a Zhou ring.

$(2)\Rightarrow (1)$ Let $a\in R$. As $S:=R/J(R)$ is a Zhou ring, it is Zhou nil-clean. It follows by~\cite[Theorem 2.11]{Z} that
$\overline{a-a^5}\in N(R/J(R))$. As $J(R)$ is nil, we see that $a-a^5\in N(R)$.
According to~\cite[Theorem 2.11]{Z}, $R$ is Zhou nil-clean, as asserted.\end{proof}

We are ready to prove the following.

\begin{thm} Let $R$ be a 2-primal Zhou nil-clean ring. Then every square matrix over $R$ is the sum of two tripotents and a nilpotent.
\end{thm}
\begin{proof} In view of~\cite[Theorem 2.11]{Z}, $R\cong S, T$ or $S\times T$, where $S$ is a 2-primal strongly nil-clean ring, and $T$ is a Zhou ring with $\frac{1}{2}\in T$.

Step 1. Let $A\in M_n(S)$. In view of ~\cite[Theorem 6.1]{KWZ}, $A$ is the sum of an idempotent and a nilpotent.

Step 2. Let $A\in M_n(T)$. In view of~\cite[Theorem 2.11]{Z}, $T/J(T)$ is commutative; and so $N(T)\subseteq J(T)$. Thus, $J(T)=N(T)$.
Then $\overline{A}\in M_n(T/J(T))$. By virtue of Theorem 2.4, there exist tripotents $\overline{C},\overline{D}\in M_n(T/J(T))$ and a nilpotent $\overline{W}\in M_n(T/J(T))$ such that $\overline{B}=\overline{C+D+W}$. Clearly, $M_n(R/J(R))\cong M_n(R)/J(M_n(R))$. As $R$ is 2-primal, it follows by~\cite[Theorem 10.21]{L} that $J(M_n(T))=M_n(J(T))=M_n(N(T))=M_n(N_*(T))=N_*(M_n(T))$ is nil, where $N_*(T)$ is the Baer-McCoy radical of $T$. As $\frac{1}{2}\in M_n(T)$, by~\cite[Lemma 2.6]{Z}, we may assume that $B,C\in M_n(T)$ are tripotents. Clearly, $W\in M_n(T)$ is nil. Thus, we can find $V\in J(M_n(T))$ such that $A=B+C+(W+V)$. Write $W^k=0$. Then $(W+V)^k\in J(M_n(T))$, and so $(W+V)^{kl}=0$ for some $l\in {\Bbb N}$. Therefore $A$ is the sum of two tripotents and a nilpotent.

Let $A\in M_n(R)$. We may write $A=(A_1,A_2)$ in $M_n(S)\times M_n(T)$, where $A_1\in M_n(S), A_2\in M_n(T)$. Thus, by the preceding discussion,
we obtain the result.\end{proof}

\begin{cor} Let $R$ be a commutative Zhou nil-clean ring. Then every square matrix over $R$ is the sum of two tripotents and a nilpotent.
\end{cor}
\begin{proof} This is obvious by Theorem 3.2, as every commutative ring is 2-primal.\end{proof}

\begin{cor} Let $R$ be a Zhou nil-clean ring, and let $A\in M_n(R)$ with central entries. Then $A$ is the sum of two tripotents and a nilpotent.
\end{cor}
\begin{proof} Let $C(R)$ be the center of $R$, and let $A\in C(R)$. In view of~\cite[Theorem 2.11]{Z}, $a-a^2\in N(R)$; and so $a-a^2\in U(C(R))$. By using~\cite[Theorem 2.11]{Z} again, $C(R)$ is a commutative Zhou nil-clean. Since $A\in M_n(C(R))$, By Corollary 3.3, $A$ is the sum of two tripotents and a nilpotent in $M_n(C(R))$. This completes the proof from $M_n(C(R))\subseteq M_n(R)$.\end{proof}

\begin{exam} Let $n\geq 2$ be an integer. If $n=2^k3^l5^m$, then every square matrix over ${\Bbb Z}_n$ is the sum of two tripotents and a nilpotent.
\end{exam}
\begin{proof} As $n=2^k3^l5^m$. It follows that ${\Bbb Z}_n \cong R_1\times R_2\times R_3$, where $R_1= {\Bbb Z}_{2^k}, R_2={\Bbb Z}_{3^l}$ and $R_3={\Bbb Z}_{5^m}$. It is obvious that each $J(R_i)$ is nil and $R_1/J(R_1)$ is a Boolean ring, $R_2/J(R_2)$ is a Yaqub ring and $R_3/J(R_3)$ is a Bell ring. It follows by~\cite[Theorem 2.11]{Z}, $R$ is a commutative Zhou nil-clean ring. Therefore we complete the proof by Corollary 3.3.\end{proof}

\begin{lem} (see~\cite[Lemma 6.6]{KWZ}). Let $R$ be of bounded index. If $J(R)$ is nil, then $J(M_n(R))$ is nil for all $n\in {\Bbb N}$.
\end{lem}

\begin{thm} Let $R$ be a Zhou nil-clean ring of bounded index. Then every square matrix over $R$ is the sum of two tripotents and a nilpotent.
\end{thm}
\begin{proof} As in the proof of Theorem 3.3, $R=S\times T$, where $S$ is strongly nil-clean of bounded index and $T$ is a Zhou nil-clean ring with $\frac{1}{2}\in S$.

Step 1. Let $A\in M_n(S)$. In view of~\cite[Corollary 6.8]{KWZ}, $A$ is the sum of an idempotent and a nilpotent. Hence, it is the sum of two tripotents and a nilpotent.

Step 2. Let $A\in M_n(T)$. Then $\overline{A}\in M_n(T/J(T))$. In light of Lemma 3.1, $T/J(T)$ is a Zhou ring. By virtue of Theorem 2.4, there exist tripotents $\overline{C},\overline{D}\in M_n(T/J(T))$ and a nilpotent $\overline{W}\in M_n(T/J(T))$ such that $\overline{A}=\overline{C+D+W}$. Clearly, $M_n(T/J(T))\cong M_n(T)/J(M_n(T))$. In light of Lemma 3.6, $J(M_n(T))$ is nil. Since $\frac{1}{2}\in M_n(T)$, it follows by~\cite[Lemma 2.6]{Z} that we may assume that $C,D\in M_n(R)$ are tripotents. Clearly, $W\in M_n(R)$ is nil. Thus, we can find $V\in J(M_n(R))$ such that $A=B+C+(W+V)$. Write $W^k=0$. Then
  $(W+V)^{kl}=0$ for some $l\in {\Bbb N}$. This shows that $A$ is the sum of two tripotents and a nilpotent.

Let $A\in M_n(R)$. Then $A=(A_1,A_2)$ with $A_1\in M_n(S)$ and $A_2\in M_n(T)$. By the discussion in Step 1 and Step 2, we complete the proof.\end{proof}

\begin{exam} Let $n\geq 2$ be an integer, and let $S=T_s({\Bbb Z}_n)(s\in {\Bbb N})$. If $n=2^k3^l5^m$, then every square matrix over $S$ is the sum of two tripotents and a nilpotent.\end{exam}
\begin{proof} As in the proof of Example 3.5, ${\Bbb Z}_n$ is a Zhou nil-clean ring. Let $A\in S$. Then $A-A^2\in N(S)$, as its diagonal entries are nilpotent. It follows by~\cite[Theorem 2.11]{Z},
$S$ is a Zhou nil-clean ring. Clearly, $S$ is of bounded index. Therefore we complete the proof, by Theorem 3.7.\end{proof}

\end{document}